\documentclass[letterpaper]{article}

\usepackage{graphicx}
\usepackage{enumerate}
\usepackage{amsmath,amsthm}
\usepackage{fullpage}
\usepackage[english]{babel}
\usepackage{amssymb}
\usepackage{ifpdf}
\usepackage{color}
\usepackage{todonotes}

\newtheorem{theorem}{Theorem}

\newtheorem{proposition}[theorem]{Proposition}

\newtheorem{lemma}[theorem]{Lemma}
\newtheorem{claim}[theorem]{Claim}

\newcommand{\blocks}{b}
\newcommand{\calG}{\mathcal{G}}
\newcommand{\calE}{\mathcal{E}}
\newcommand{\calC}{\mathcal{C}}
\newcommand{\calP}{\mathcal{P}}
\newcommand{\NN}{\mathbb{N}}
\newcommand{\indic}{\mathbb{I}}

\newcommand{\eqdef}{\stackrel{\mathrm{def}}{=}}

\newcommand{\expec}[2]{\mathbb{E}_{#1}\left(#2\right)}
\newcommand{\prob}[2]{\mathbb{P}_{#1}\left(#2\right)}
\newcommand{\setof}[1]{\{#1\}}

\newcommand{\LCS}{\mbox{\textsc{LCS}}}
\newcommand{\rflcs}{\mbox{\textsc{RFLCS}}}
\newcommand{\opt}{\ensuremath{\mathrm{Opt}}}
\def\eps{\epsilon}

\begin{document}

\title{Repetition-free longest common subsequence of random sequences} 



\author{Cristina G.~Fernandes
\thanks{Computer Science Department, Universidade de S\~ao Paulo, Brazil.
  Web: \texttt{www.ime.usp.br/$\sim$cris}. Partially supported by 
  CNPq 308523/2012-1, 477203/2012-4 and Proj.\ MaCLinC of NUMEC/USP.}
 \and 
  Marcos Kiwi 
\thanks{
  Depto.~Ing.~Matem\'{a}tica \&
  Ctr.~Modelamiento Matem\'atico UMI 2807, U.~Chile.
  Web: \texttt{www.dim.uchile.cl/$\sim$mkiwi}.
  Gratefully acknowledges the support of 
    Millennium Nucleus Information and Coordination in Networks ICM/FIC P10-024F
    and CONICYT via Basal in Applied Mathematics.}
}

\date{\today}

\maketitle

\begin{abstract}
A repetition free Longest Common Subsequence (LCS) of two sequences 
  $x$ and $y$ is an LCS of $x$ and~$y$ where each symbol may appear
  at most once.
Let $R$ denote the length of a repetition free 
  LCS of two sequences of $n$ symbols each one 
  chosen randomly, uniformly, and independently
  over a $k$-ary alphabet. 
We study the asymptotic, in $n$ and $k$, behavior of $R$ and 
  establish that there are three distinct regimes, depending 
  on the relative speed of growth of $n$ and $k$.
For each regime we establish the limiting behavior of $R$. 
In fact, we do more, since we actually establish tail bounds for large 
  deviations of $R$ from its limiting behavior. 

Our study is motivated by the so called exemplar model proposed by 
  Sankoff (1999) and the related similarity measure introduced by 
  Adi~et~al.~(2007).
A natural question that arises in this context, which as we show
  is related to long standing open problems in the area of probabilistic
  combinatorics, is to understand the asymptotic, in $n$ and $k$, 
  behavior of parameter $R$.
\end{abstract}

\section{Introduction}\label{intro}
Several of the genome similarity measures considered in the literature
either assume that the genomes do not contain gene duplicates, or work
efficiently only under this assumption. However, several known genomes
do contain a significant amount of duplicates. (See the review on
gene and genome duplication by Sankoff~\cite{Sankoff01} for specific
information and references.) 
One can find in the literature proposals to address this issue.
Some of these proposals suggest to filter the genomes, throwing away
part or all of the duplicates, and then applying the desired similarity
measure to the filtered genomes. (See~\cite{AngibaudFRTV09} for a
description of different similarity measures and filtering models for
addressing duplicates.)

Sankoff~\cite{Sankoff99}, trying to take into account gene duplication
in genome rearrangement, proposed the so called exemplar model, which
is one of the filtering schemes mentioned above. In this model, one
searches, for each family of duplicated genes, an exemplar
representative in each genome. Once the representative genes are
selected, the other genes are disregarded, and the part of the genomes
with only the representative genes is submitted to the similarity
measure. In this case, the filtered genomes do not contain duplicates,
therefore several of the similarity measures (efficiently) apply. Of
course, the selection of the exemplar representative of each gene
family might affect the result of the similarity measure. Following the
parsimony principle, one wishes to select the representatives in such
a way that the resulting similarity is as good as possible. Therefore,
each similarity measure induces an optimization problem: how to select
exemplar representatives of each gene family that result in the best
similarity according to that specific measure.

The length of a Longest Common Subsequence (\LCS) is a well-known measure of
similarity between sequences. In particular, in genomics, the length of an
\LCS\ is directly related to the so called edit distance between two sequences
when only insertions and deletions are allowed, but no substitution. This
similarity measure can be computed efficiently 
in the presence of
duplicates (the classical dynamic programming solution to the LCS problem
  takes quadratic time, however, improved algorithms are known, specially
  when additional complexity parameters are taken into account -- for a 
  comprehensive comparison of well-known algorithms for the LCS problem,
  see~\cite{BHR00}). 
Inspired by the exemplar model above, some variants of
the LCS similarity
  measure have been proposed in the literature. One of them, the so called
\emph{exemplar} \LCS~\cite{BonizzoniDVDFRV07}, uses the concept of mandatory and
optional symbols, and searches for an \LCS\ containing all mandatory
symbols. A second one is the so called \emph{repetition-free}
\LCS~\cite{AdiBFFMSSTW07}, that requires each symbol to appear at most once in
the subsequence. Some other extensions of these two measures were considered
under the name of \emph{constrained} \LCS\ and \emph{doubly-constrained}
\LCS~\cite{BonizzoniDVDP10}. All of these variants were shown to be hard to
compute~\cite{AdiBFFMSSTW07,BlinBDS12,BonizzoniDVDFRV07,BonizzoniDVDP10}, so some
heuristics and approximation algorithms for them were proposed and
experimentally tested~\cite{AdiBFFMSSTW07,BonizzoniDVDFRV07}.

Specifically, the notion of repetition-free \LCS\ was formalized by Adi et
al.~\cite{AdiBFFMSSTW07} as follows.  They consider finite sets, called
\emph{alphabets}, whose elements are referred to as \emph{symbols}, and then
they define the \rflcs\ problem as: Given two sequences $x$ and~$y$, find a
repetition-free \LCS\ of $x$ and $y$.  We write \rflcs$(x,y)$ to refer to the
\rflcs\ problem for a generic instance consisting of a pair~$(x,y)$, and we
denote by $\opt(\rflcs(x,y))$ the length of an optimal solution of
\rflcs$(x,y)$.
In their paper, Adi et al.~showed that \rflcs\ is MAX SNP-hard,
  proposed three approximation algorithms for \rflcs, and presented an
  experimental evaluation of their proposed algorithms, using 
  for the sake of comparison an exact (computationally expensive)
  algorithm for \rflcs\ based on an integer linear programming
  formulation of the problem. 

Whenever a problem such as the \rflcs\ is considered, a very 
  natural question arises: What is the expected value of 
  $\opt(\rflcs(x,y))$? (where expectation is taken over
  the appropriate distribution over the instances $(x,y)$
  one is interested in).
It is often the case that one has little knowledge of the distribution 
  of problem instances, except maybe for the size of the instances.
Thus, an even more basic and often relevant issue is 
  to determine the expected value taken by $\opt(\rflcs(x,y))$
  for uniformly distributed choices of~$x$ and~$y$ over all 
  strings of a given length over some fixed size alphabet 
  (say each sequence has $n$ symbols randomly, uniformly, and 
  independently chosen over a $k$-ary alphabet $\Sigma$).
Knowledge of such an average case behavior is a first step in
  the understanding of whether a specific value of 
  $\opt(\rflcs(x,y))$ is of relevance or could be simply explained 
  by random noise. 
The determination of this later average case behavior in the 
  asymptotic regime (when the length $n$ of the sequences $x$ and $y$ go
  to infinity) is the main problem  we undertake in this work.
Specifically, let $R_n=R_n(x,y)$ denote the length 
  of a repetition-free LCS of
  two sequences~$x$ and $y$ of~$n$ symbols randomly, uniformly, and
  independently chosen over a $k$-ary alphabet.  Note that the 
  random variable~$R_n$ is simply the
  value of $\opt(\rflcs(x,y))$.
We are interested in determining (approximately) the value of $\expec{}{R_n}$ 
  as a function of $n$ and $k$, for very large values of $n$.

Among the results established in this work, is that 
  the behavior of $\expec{}{R_n}$ depends on 
  the way in which $n$ and $k$ are related.
In fact, if $k$ is fixed, it is easy to see that $\expec{}{R_n}$ tends to
  $k$ when $n$ goes to infinity (simply because any fix permutation 
  of a $k$-ary alphabet will appear in a sufficiently large sequence
  of uniformly and independently chosen symbols from the alphabet).
Thus, the interesting cases arise when $k=k(n)$ tends to infinity with $n$.
However, the speed at which $k(n)$ goes to infinity is of crucial relevance 
  in the study of the behavior of $\expec{}{R_n}$.
This work identifies three distinct growth regimes depending on the 
  asymptotic dependency between $n$ and $k\sqrt{k}$.
Specifically, our work establishes the next result:
\begin{theorem}\label{th:main}
The following holds:
\begin{itemize}
\item If $n=o(k\sqrt{k})$, then 
  $\displaystyle\lim_{n\to\infty}\frac{\expec{}{R_n}}{n/\sqrt{k(n)}} = 2$.

\item If $n=\frac{1}{2}\rho k\sqrt{k}$ for $\rho> 0$, then 
  $\displaystyle\liminf_{n\to\infty}\frac{\expec{}{R_n}}{k(n)} \geq 1-e^{-\rho}$.
  (By definition $R_n\leq k(n)$.)

\item If $n=(\frac{1}{2}+\xi)k\sqrt{k}\ln k$ for some $\xi>0$, then
  $\displaystyle\lim_{n\to\infty}\frac{\expec{}{R_n}}{k(n)} = 1$.

\end{itemize}
\end{theorem}
In fact, we do much more than just proving the preceding result.
Indeed, for each of the three different regimes of Theorem~\ref{th:main} 
  we establish so called large deviation bounds which capture how
  unlikely it is for $R_n$ to deviate too much from its expected value.
We relate the asymptotic average case behavior of $\expec{}{R_n}$ 
  with that of the length $L_n=L_n(x,y)$ 
  of a Longest Common Subsequence (LCS) of two sequences $x$ and $y$
  of~$n$ symbols chosen randomly, uniformly, and independently
  over a $k$-ary alphabet. 
A simple (well-known) fact concerning~$L_n$ is that 
  $\expec{}{L_n}/n$ tends to a constant, say $\gamma_k$, when $n$ goes 
  to infinity.
The constant $\gamma_k$ is known as the Chv\'atal-Sankoff constant.
A long standing open problem is to determine the exact value of~$\gamma_k$ for 
  any fixed~${k\geq 2}$.
However, Kiwi, Loebl, and Matou\v{s}ek~\cite{KiwiLM05} proved that 
  $\gamma_k\sqrt{k}\rightarrow 2$ as $k \rightarrow \infty$
  (which positively settled a conjecture due to Sankoff and 
  Mainville~\cite{SankoffM83}). 
In the derivation of Theorem~\ref{th:main} we build 
  upon~\cite{KiwiLM05}, and draw connections 
  with another intensively studied problem concerning Longest 
  Increasing Subsequences (LIS) of randomly chosen permutations (also
  known as Ulam's problem).
Probably even more significant is the fact that our work partly elicits a 
  typical structure of one of the large repetition-free common subsequences
  of two length $n$ sequences randomly, uniformly, and 
  independently chosen over a $k$-ary alphabet.



Before concluding this introductory section, we
  discuss a byproduct of our work. 
To do so, we note that the computational experiments presented by Adi et
  al.~\cite{AdiBFFMSSTW07} considered problem instances 
  where sequences of~$n$ symbols where randomly,
  uniformly, and independently chosen over a $k$-ary alphabet.  
The experimental findings
  are consistent with our estimates of $\expec{}{R_n}$.
Our results thus have the added bonus,
  at least when~$n$ and $k$ are large, that they allow to perform 
  comparative studies, as the aforementioned one,
  but replacing the (expensive) exact computation of 
  $R_n$ by our estimated value.
Our work also suggests that 
  additional experimental evaluation of proposed heuristics,
  over test cases generated as in so
  called \emph{planted random models}, might help to further validate the
  usefulness of proposed algorithmic approaches.
Specifically, for the \rflcs\ problem, according to the
  planted random model, one way to generate test cases would
  be as described next. 
First, for some fixed $\ell\leq k$, 
  choose a repetition-free sequence~$z$ of length
  $\ell<n$ over a $k$-ary alphabet.
Next, generate a
  sequence~$x'$ of~$n$ symbols randomly, uniformly, and
  independently over the $k$-ary alphabet. 
Finally, uniformly at random choose a size $\ell$ collection 
  $s_1,\ldots,s_{\ell}\subseteq \{1,\ldots,n\}$ of 
  distinct positions of~$x'$ and replace the $s_i$th symbol
  of~$x'$ by the $i$th symbol of $z$, thus ``planting'' $z$ in $x'$.
Let $x$ be the length~$n$ sequence thus obtained.
Repeat the same procedure again for a second sequence $y'$ also
  of $n$ randomly chosen symbols but with 
  the same sequence $z$, and obtain a new sequence $y$. 
The resulting sequences~$x$ and~$y$ are such that 
  $\rflcs(x,y) \geq \ell$.
The parameter $\ell$ can be chosen to be larger than the 
  value our work predicts for $R_n$.
This allows to efficiently generate ``non typical'' problem 
  instances over which to try out the heuristics, as well 
  as a lower bound certificate for the problem optimum (although,
  not a matching upper bound).
For more details on the planted random model the interested reader is
  referred to the work of Bui, Chaudhuri, Leighton, and Sipser~\cite{BCLS87},
  where (to the best of our knowledge) the model first appeared, and to 
  follow up work by Boppana~\cite{Bop87}, 
  Jerrum and Sorkin~\cite{JS93},
  Condon and Karp~\cite{CK99}, 
  and the more recent work of Coja-Oghlan~\cite{Coj06}.
 
\medskip
Next, we formalize some aspects of our preceding discussion
  and rigorously state and derive our claims.
However, we first need to introduce terminology, 
  some background material, and establish some basic facts.
We start by describing the road-map followed throughout this manuscript.

\medskip\medskip\noindent
\textbf{Organization: }
This work is organized as follows. 
In Section~\ref{sec:urns}, we review some classical probabilistic 
  so called urn models and, for the sake of completeness,
  summarize some of their known basic properties, as well as 
  establish a few others.
As our results build upon those of Kiwi, Loebl, and 
  Matou\v{s}ek~\cite{KiwiLM05}, we review them in Section~\ref{sec:lcs}, 
  and also take the opportunity to introduce some relevant terminology. 
In Section~\ref{sec:distrib}, we formalize the notion of 
  ``canonical'' repetition-free LCS and show that conditioning on its size, 
  the distribution of the set of its symbols is uniform 
  (among all appropriate size subsets of symbols).
Although simple to establish, this result is key to our approach since it 
  allows us to relate the probabilistic analysis of the length of 
  repetition-free LCSs to one concerning urn models.
Finally, in Section~\ref{sec:lower}, we establish large deviation 
  type bounds from which Theorem~\ref{th:main} easily follows.

\section{Background on urn models}\label{sec:urns}
The probabilistic study of repetition-free LCSs we will undertake 
  will rely on the understanding of random phenomena 
  that arises in so called urn models. 
In these models, there is a collection of urns where 
  balls are randomly placed. 
Different ways of distributing the balls in the urns, as well
  as considerations about the (in)distinguishability of urns/balls,
  give rise to distinct models, often referred to in the literature 
  as occupancy problems  (for a classical treatment 
  see~\cite{feller}). 
In this section, we describe those urn models we will later encounter, 
  associate to them parameters of interest, and state some basic
  results concerning their probabilistic behavior.

Henceforth, let $k$ and $s$ be positive integers, and 
  $\vec{s}=(s_1,\ldots,s_{\blocks})$ denote 
  a $\blocks$-dimensional nonnegative integer vector whose
  coordinates sum up to $s$, i.e.~$\sum_{i=1}^{\blocks} s_i=s$.
For a positive integer $m$, we denote the set 
  $\setof{1,\ldots,m}$ by $[m]$.

Consider the following two processes where $s$ indistinguishable balls
  are randomly distributed among $k$ distinguishable urns.
\begin{itemize}
\item 
\textbf{Grouped Urn $(k,\vec{s})$-model:}
Randomly distribute $s$ balls over $k$ urns, placing a ball in urn 
  $j$ if $j\in S_i$, where  
  $S_{1},\ldots,S_{\blocks}\subseteq [k]$ are chosen randomly and
  independently so that 
  $S_{i}$ is uniformly distributed among all subsets of $[k]$ of size $s_{i}$.

\item
\textbf{Classical Urn $(k,s)$-model:}
Randomly distribute $s$ balls over $k$ urns, so that the urn 
  on which the $i$th ball, $i\in [k]$, is placed is uniformly chosen
  among the $k$ urns, and independently of where the other balls 
  are placed.\footnote{Note that this model is a particular case 
  of the Grouped Urn model where $\blocks=s$ and $s_1=\cdots=s_{\blocks}=1$.}
\end{itemize}
Henceforth, let $X^{(k,\vec{s})}$ be the number of empty urns left 
  when the Grouped Urn $(k,\vec{s})$-process ends.
Furthermore, let $X^{(k,\vec{s})}_{j}$ be the indicator of the event that the 
  $j$th urn ends up empty.
Obviously, $X^{(k,\vec{s})}=\sum_{j=1}^{k}X^{(k,\vec{s})}_{j}$.
Similarly, define $Y^{(k,s)}$ and $Y^{(k,s)}_1,\ldots,Y^{(k,s)}_{k}$ 
  but with respect to the Classical Urn $(k,s)$-process.
Intuitively, one expects that fewer urns will end up empty in the 
  Grouped Urn process in comparison with the Classical Urn process.
This intuition is formalized through the following result.
\begin{lemma}\label{lem:dominance}
Let $\vec{s}=(s_1,\ldots,s_{\blocks})\in\NN^{b}$ and 
  $s=\sum_{i=1}^{\blocks} s_{i}$.
Then, the random variable $X^{(k,\vec{s})}$ dominates 
  $Y^{(k,s)}$, i.e.~for every $t\geq 0$, 
\[
\prob{}{X^{(k,\vec{s})}\geq t} \ \leq \ \prob{}{Y^{(k,s)}\geq t}.
\]
\end{lemma}
\begin{proof}
First observe that if $\vec{s}=(1,\ldots,1)\in\NN^{s}$, then
  $X^{(k,\vec{s})}$ and $Y^{(k,s)}$ have the same distribution, thence
  the claimed result trivially holds for such $\vec{s}$.
For $\vec{s}=(s_1,\ldots,s_{\blocks})\in\NN^{\blocks}$ with $\sum_{i=1}^{\blocks} s_i=s$ 
  and $s_j \geq 2$ for some~$j\in [\blocks]$, let 
  $\vec{s'}=(s'_1,\ldots,s'_{\blocks},1)\in\NN^{\blocks+1}$ be such that
\[
\vec{s'}=(s_1,\ldots,s_{j-1},s_{j}-1,s_{j+1},\ldots, s_{\blocks},1).
\]
Note that $\sum_{i=1}^{\blocks +1}s'_i=s$ and 
  observe that, to establish the claimed result, it will be enough
  to inductively show that, for every $t\geq 0$,
\begin{align}\label{eqn:dominate}
\prob{}{X^{(k,\vec{s})}\geq t} \ \leq \ \prob{}{X^{(k,\vec{s'})}\geq t}.
\end{align}
To prove this last inequality, consider the following experiment.
Randomly choose $S_1,\ldots,S_{\blocks}$
  as in the Grouped Urn $(k,\vec{s})$-model described above, and 
  distribute $s$ balls in $k$ urns as suggested in the model's description.
Recall that $X^{(k,\vec{s})}$ is the number of empty urns left when
  the process ends.
Now, randomly and uniformly choose one of the balls placed
  in an urn of index in $S_j$. With probability $\frac{k-(s_j-1)}k$, 
  leave it where it is and, with probability $\frac{s_j-1}k$, 
  move it to a distinct urn of index in $S_j$ chosen randomly and uniformly.
Observe that the number of empty urns cannot decrease.
Moreover, note that the experiment just described is equivalent 
  to the Grouped Urn $(k,\vec{s'})$-model, thence the number of 
  empty urns when the process ends is distributed according to 
  $X^{(k,\vec{s'})}$.
It follows that~\eqref{eqn:dominate} holds, thus concluding the proof
  of the claimed result. 
\end{proof}

We will later need upper bounds on
  the probability that a random variable distributed 
  as $X^{(k,\vec{s})}$ 
  is bounded away (from below) from its expectation, i.e.~on so called
  upper tail bounds for $X^{(k,\vec{s})}$.
The relevance of Lemma~\ref{lem:dominance} 
  is that it allows us to concentrate on the rather more
  manageable random variable $Y^{(k,s)}$, 
  since any upper bound on the probability that 
  $Y^{(k,s)}\geq t$ will also be valid for the probability that 
  $X^{(k,\vec{s})}\geq t$.
The behavior of $Y^{(k,s)}$ is a classical thoroughly studied subject.
In particular,
  there are well-known tail bounds that apply to it. 
A key fact used in the derivation of such tail bounds is that 
  $Y^{(k,s)}$ is the sum of the 
  negatively related $0$-$1$ random
  variables $Y^{(k,s)}_1, \ldots, Y^{(k,s)}_{k}$ 
  (for the definition of negatively related random variables
  see~\cite{janson94}, and the discussion in~\cite[Example~1]{janson94}).
For convenience of future 
  reference, the next result summarizes the tail bounds
  that we will use.
\begin{proposition}\label{prop:regimes}
For all positive integers $k$ and $s$,
\begin{align}\label{eqn:lambda}
\lambda \eqdef \expec{}{Y^{(k,s)}} = k\left(1-\frac{1}{k}\right)^{s}.
\end{align}
Moreover, the following hold:
\begin{enumerate}
\item 
\label{it:middle-regime}
If $p\eqdef \lambda/k$ and $q\eqdef 1-p$, then for all $a\geq 0$,
\[
\prob{}{Y^{(k,s)}\geq \lambda+a} 
  \leq \exp\left(-\frac{a^2}{2(kpq+a/3)}\right).
\]

\item 
\label{it:large-regime}
Let $\xi\geq 0$ and $s=(1+\xi)k\ln k$, then 
\[
\prob{}{Y^{(k,s)}\neq 0} \leq \frac{1}{k^{\xi}}.
\]

\item
\label{it:small-regime}
For all $a> 0$,
\[
\prob{}{k-Y^{(k,s)}\leq s-a} \leq \left(\frac{es^2}{ka}\right)^{a}.
\]
\end{enumerate}
\end{proposition}
\begin{proof}
Since the probability that a ball uniformly distributed over
  $k$ urns lands in urn $j$ is $1/k$, the probability that
  none of $s$ balls lands in urn $j$ (equivalently,
  that $Y^{(k,s)}=1$) is exactly $(1-1/k)^{s}$.
By linearity of expectation, to establish~\eqref{eqn:lambda},
  it suffices to observe that
  $\expec{}{Y^{(k,s)}}=\sum_{j=1}^{k}\prob{}{Y^{(k,s)}_{j}=1}$.

Part~\ref{it:middle-regime} is just a re-statement of 
  the second bound in~(1.4) of~\cite{janson94} taking into account 
  the comments in~\cite[Example~1]{janson94}.

Part~\ref{it:large-regime} is a folklore result that follows 
  easily from an application of the union bound. 
For completeness, we sketch the proof.
Note that $Y^{(k,s)}\neq 0$ if and only if $Y_{j}^{(k,s)}\neq 0$
  for some $j\in [k]$.
Hence, by a union bound and since $1-x\leq e^{-x}$ for all $x$,
\[
\prob{}{Y^{(k,s)}\neq 0} \leq 
  \sum_{j\in [k]}\prob{}{Y_{j}^{(k,s)}\neq 0}
  = k\left(1-\frac{1}{k}\right)^{s}
  \leq ke^{-s/k}
  = \frac{1}{k^{\xi}}.
\]

Finally, let us establish Part~\ref{it:small-regime}.
Observe that $k-Y^{(k,s)}$ is the number of urns that end up nonempty
  in the Classical Urn $(k,s)$-model.
Thus, assuming that balls are sequentially thrown, one by one, 
  if $k-Y^{(k,s)}\leq s-a$, then there must be a 
  size $a$ subset $S\subseteq [s]$ of balls that fall in an urn where a 
  previously thrown ball has already landed.
The probability that a ball in $S$ ends up in a previously 
  occupied urn, is at most $s/k$ (given that at any moment
  at most $s$ of the $k$ urns are occupied).
So the probability that all balls in $S$ end up in 
  previously occupied urns is at most $(s/k)^{a}$.
Thus, by a union bound, some algebra, and the standard bound on 
  binomial coefficients $\binom{\mu}{\nu}\leq (e\mu/\nu)^{\nu}$,
\begin{align*}
\prob{}{k-Y^{(k,s)} \leq s-a}
  & 
  \leq \sum_{S\subseteq [s]: |S|=a} 
            \left(\frac{s}{k}\right)^{a}
    \leq \binom{s}{a}\left(\frac{s}{k}\right)^{a}
    \leq \left(\frac{es^{2}}{ka}\right)^{a}. \qedhere
\end{align*}
\end{proof}

\section{Some background on the expected length of an LCS}\label{sec:lcs}
In~\cite{KiwiLM05}, pairs of sequences $(x,y)$ are associated to 
  plane embeddings of bipartite graphs, and a common subsequence 
  of $x$ and $y$ to a special class of matching of the associated 
  bipartite graph.
Adopting this perspective will also be useful in this work.
In this section, 
  besides reviewing and restating some of the results of~\cite{KiwiLM05},
  we will introduce some of the terminology we shall adhere in
  what follows. 
  

The \emph{random word model $\Sigma(K_{r,s};k)$},\footnote{Remember 
  that~$K_{r,s}$ denotes the complete bipartite graph with two 
  bipartition classes, one of size $r$ and the other of size~$s$.}
  as introduced
  in~\cite{KiwiLM05}, consists of the
  following (for an illustration, see Figure~\ref{fig:model}): 
  the distribution over the set of subgraphs of $K_{r,s}$
  obtained by uniformly and independently assigning to each vertex of
  $K_{r,s}$ one of $k$ symbols and keeping those edges whose
  endpoints are associated to the same symbol.

\begin{figure}[h]
\begin{center}
\ifpdf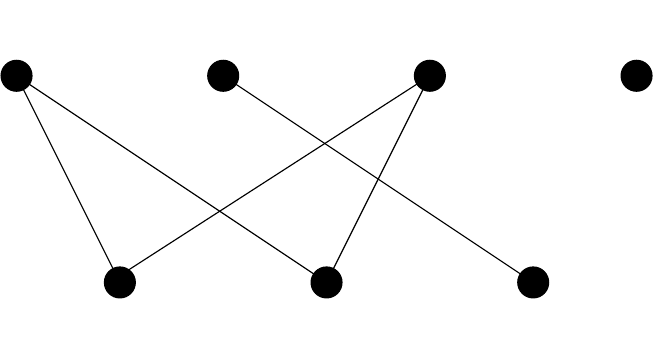\else\input{randomWordModel.pstex_t}\fi
\end{center}
\caption{Graph obtained from $\Sigma(K_{4,3};3)$ for the 
  choice of symbols associated (shown close to) each node.}\label{fig:model}
\end{figure}

Following~\cite{KiwiLM05}, two distinct edges $ab$ and $a'b'$ of $G$ are
said to be \emph{noncrossing} if $a$ and $a'$ are in the same order as
$b$ and $b'$. In other words, if $a<a'$ and $b<b'$, or $a'<a$ and
$b'<b$. A matching of $G$ is called \emph{noncrossing} if every distinct
pair of its edges is noncrossing.

Henceforth, for 
  a bipartite graph $G$, we denote by $L(G)$ the number of edges in
a maximum size (largest) noncrossing matching of $G$. 
To $G$  chosen according to $\Sigma(K_{n,n};k)$ we can associate two
  sequences of length~$n$, say $x(G)$ and $y(G)$,  
  one for each of the bipartition sides of~$G$, consisting
  of the symbols associated to the vertices of $K_{n,n}$. 
Note that $x(G)$ and $y(G)$ are uniformly and independently distributed
  sequences of~$n$ symbols over a~$k$-ary alphabet. 
Observe that, if $G$ is chosen according to $\Sigma(K_{n,n};k)$, then
  $L(G)$ is precisely the length of an LCS of its two associated
  sequences $x(G)$ and $y(G)$, and vice versa. 
Formally, $L(G)=L_{n}(x(G),y(G))$, where $L_n(\cdot,\cdot)$ 
  is as defined in the introductory section.

Among other things, in~\cite{KiwiLM05}, it is shown that 
  $L(\Sigma(K_{n,n};k))\sqrt{k}/n$ is approximately equal to $2$ 
  when $n$ and $k$ are very large, 
  provided that $n$ is ``sufficiently large'' compared to $k$.
This result is formalized in the following:
\begin{theorem}[Kiwi, Loebl, and Matou\v sek~\cite{KiwiLM05}]
\label{ThmL}
For every $\eps>0$, there exist $k_0$ and~$C$ such that, for all
$k>k_0$ and all $n$ with $n > C\sqrt{k}$, 
\begin{align}\label{th:prev-lower-bound}
(1-\eps)\cdot\frac{2n}{\sqrt{k}} &
  \ \leq \ \expec{}{L(\Sigma(K_{n,n};k))} 
  \ \leq \ (1+\eps)\cdot\frac{2n}{\sqrt{k}}.
\end{align}
Moreover, there is an exponentially small tail bound; namely, for
every $\eps>0$, there exists $c>0$ such that for $k$ and $n$ as above, 
$$ 
\prob{}{\left|L(\Sigma(K_{n,n};k))-\frac{2n}{\sqrt{k}}\right| 
  \geq \eps\frac{2n}{\sqrt{k}}}
   \ \leq \ e^{-cn/\sqrt{k}}. $$
\end{theorem}
Observe now that a graph $G$ chosen according 
  to $\Sigma(K_{n,n};k)$ has symbols, from a $k$-ary alphabet,
  implicitly associated to each of its nodes (to the $j$th node
  of each side of $G$ the $j$th symbol of the corresponding sequence~$x(G)$ or $y(G)$).
Furthermore, the endpoints of an edge $e$ of $G$ must, by construction,
  be associated to the same symbol, henceforth referred to the symbol
  associated to $e$.
We say that a noncrossing matching of $G$ is \emph{repetition-free}
  if the symbols associated to its edges are all distinct, and we 
  denote by $R(G)$ the number of edges in a maximum size
  (largest) repetition-free noncrossing matching of $G$.
If $G$ is chosen according to $\Sigma(K_{n,n};k)$, then
  $R(G)$ is precisely the length of a repetition-free LCS of its two associated
  sequences $x(G)$ and $y(G)$, and vice versa. 
Formally, $R(G)=R_{n}(x(G),y(G))$, where again 
  $R_n(\cdot,\cdot)$ is as defined in the introductory section.
Summarizing, we have reformulated the repetition-free LCS problem
  as an equivalent one, but concerning repetition-free noncrossing matchings.
This justifies why, from now on, we will speak interchangeably about
  repetition-free LCSs and repetition-free noncrossing matchings.

Clearly, for every $G$ in the support of $\Sigma(K_{n,n};k)$
  we always have that $R(G)\leq L(G)$.
So, the upper bound in~\eqref{th:prev-lower-bound} for 
  $\expec{}{L(\Sigma(K_{n,n};k))}$ and the upper tail bound for 
  $L(\Sigma(K_{n,n};k))$ of Theorem~\ref{ThmL} are valid
  replacing $L(\Sigma(K_{n,n};k))$ by $R(\Sigma(K_{n,n};k))$.
This explains why from now on we concentrate exclusively in the 
  derivation of lower bounds such as those of Theorem~\ref{ThmL} but
  concerning $R(\cdot)$.

Our approach partly builds on~\cite{KiwiLM05}, so
  to help the reader follow the rest of this work, it will 
  be convenient to have a high level understanding of 
  the proofs of the lower bounds in Theorem~\ref{ThmL}.
We next provide such a general overview.
For precise statements and detailed proofs, 
  see~\cite{KiwiLM05}.

The proof of the lower bound in~\eqref{th:prev-lower-bound} has two parts,
  both of which consider a graph~$G$ 
  chosen according to~$\Sigma(K_{n,n};k)$ whose 
  sides, say~$A$ and~$B$, are
  partitioned into segments $A_1,A_2,\ldots$ and $B_1,B_2,\ldots$,
  respectively, of roughly the same appropriately chosen size 
  $\widetilde{n}=\widetilde{n}(k)$.
For each~$i$, one considers the subgraph of $G$ induced by~$A_i \cup B_i$,
  say $G_i$, and observes that 
  the union of noncrossing matchings, one for each~$G_i$, is a
  noncrossing matching of~$G$. 
The first part of the proof argument 
  is a lower bound on the expected length of a 
  largest noncrossing matching of~$G_i$. 
The other part of the proof
  is a lower bound on the expected length of a largest 
  noncrossing matching of $G$ which follows  
  simply by summing the lower bounds from the first part 
  and observing that, by ``sub-additivity'', $\sum_{i} L(G_i)\leq L(G)$.

Since the size of the segments $A_1,A_2,\ldots$ and 
  $B_1,B_2,\ldots$ is~$\widetilde{n}$, 
  there are~$n/\widetilde{n}$ such segments in~$A$ and in~$B$. 
An edge of $K_{n,n}$ is in~$G$ with probability~$1/k$. 
So the expected number of edges in~$G_i$ is $\widetilde{n}^2/k$. 
The value of~$\widetilde{n}$ is chosen so that, for
  each~$i$, the expected number of edges of $G_i$ is large, 
  and the expected degree of each vertex of $G_i$ is 
  much smaller than~$1$. 
Let $G'_i$ be the graph obtained from~$G_i$ by removing isolated vertices 
  and the edges 
  incident to vertices of degree greater than~$1$. 
By the choice of $\widetilde{n}$, almost all nonisolated vertices of~$G_i$
  have degree~$1$. 
So $G'_i$ has ``almost''~the same expected number of edges as $G_i$, 
  i.e.~$\widetilde{n}^2/k$ edges.  
Also, note that~$G'_i$ is just a perfect matching (every node has degree 
  exactly~$1$).
This perfect matching, of size say~$t$, defines a permutation of~$[t]$ ---
  in fact, by symmetry arguments it is easy to see that, conditioning 
  on $t$, the permutation is uniformly distributed 
  among all permutations of $[t]$.
Observe that a noncrossing matching of~$G'_i$ corresponds to an increasing 
  sequence in the aforementioned permutation, and vice versa. 
So a largest noncrossing matching of~$G'_i$ is given by a Longest
  Increasing Sequence (LIS) of the permutation.  
There are precise results (by Baik et al.~\cite{BaikDJ99}) on the 
  distribution of the length of a LIS of a randomly chosen permutation 
  of~$[t]$.  
The expected length of a LIS for such a random permutation is 
  $2\sqrt{t}$. 
So a largest noncrossing matching in~$G'_i$ has expected length 
  almost~$2\sqrt{\widetilde{n}^2/k} = 2\widetilde{n}/\sqrt{k}$.  
As the number of $i$'s is~$n/\widetilde{n}$, we obtain a
  lower bound of 
  almost~$(n/\widetilde{n})2\widetilde{n}/\sqrt{k} = 2n/\sqrt{k}$ for
  the expected length of a largest noncrossing matching of $G$.
The same reasoning (although technically significantly 
  more involved) yields a lower tail bound for the deviation of 
  $\sum_{i} L(G'_i)\leq \sum_{i}L(G_i)\leq L(G)$ from $2n/\sqrt{k}$.
This concludes our overview of the proof arguments of~\cite{KiwiLM05}
  for deriving the lower bounds of Theorem~\ref{ThmL}.

We now stress one important aspect of the preceding paragraph discussion.
Namely, that by construction $G'_i$ is a subgraph of~$G_i$ 
  whose vertices all have degree one,
and, moreover, $G'_i$ is in fact an induced subgraph of~$G$.
Since~$G$ is generated according to $\Sigma(K_{n,n};k)$, it must 
  necessarily be the case that the symbols associated to the edges of 
  $G'_i$ are \emph{all} distinct. 
Hence, a noncrossing matching of~$G'_i$ is also a repetition-free
  noncrossing matching of~$G'_i$, thence also of~$G_i$.
In other words, it holds that $L(G'_i)=R(G'_i)\leq R(G_i)$.
Thus, a lower tail bound for the deviation 
  of $L(G'_i)$ from $2\widetilde{n}/\sqrt{k}$ 
  is also a lower tail bound for the deviation of $R(G_i)$ from 
  $2\widetilde{n}/\sqrt{k}$.
Unfortunately, $R(\cdot)$ is not sub-additive as $L(\cdot)$ above
  (so now, $\sum_i R(G_i)$ is not necessarily a lower bound for $R(G)$).
Indeed, the union of 
  repetition-free noncrossing matchings $M_i$ of the $G_i$'s is certainly
  a noncrossing matching, but is not necessarily repetition-free.
This happens because although the symbols associated to the edges of each
  $M_i$ must be distinct, it might happen that the same symbol is associated
  to several edges of different $M_i$'s.
However, if we can estimate (bound) the number of ``symbol overlaps'' 
  between edges of distinct $M_i$'s, then we can potentially translate
  lower tail bounds for the deviation of $L(G'_i)=R(G'_i)$ from
  some given value, to lower tail bounds for the deviation
  of $R(G)$ from a properly chosen factor of the given value.
This is the approach we will develop in detail in the following sections.
However, we still need a lower tail bound for $R(G_i)$ when 
  $\widetilde{n}=\widetilde{n}(k)$
  is appropriately chosen in terms of $k$ and $G_i$ is  
  randomly chosen as above.
{From} the previous discussion, it should be clear that 
  such a tail bound is implicitly established in~\cite{KiwiLM05}.
The formal result of~\cite{KiwiLM05}
  related to $G_i$, addresses the distribution of
  $L(\Sigma(K_{r,s};k))$ for $r=s=\widetilde{n}$, 
  as expressed in their Proposition 6 in~\cite[p.~486]{KiwiLM05}. 
One can verify that the same result holds, with the same proof, 
  observing that each $G_i$ is distributed according 
  to $\Sigma(K_{\widetilde{n},\widetilde{n}};k)$ and 
  replacing $L(\cdot)$ by~$R(\cdot)$.
For the sake of future reference, 
  we re-state the claimed result but with respect to 
  the parameter $R(\cdot)$ and the case $r=s=\widetilde{n}$ 
  we are interested in.
\begin{theorem}\label{theo:small-main}
For every $\delta>0$, there exists $C=C(\delta)$ such that,
  if $\widetilde{n}$ is an integer and 
  $C\sqrt{k}\leq \widetilde{n}\leq \delta k/12$ then, 
  with $m_{u}=2(1+\delta)\widetilde{n}/\sqrt{k}$ and 
  $m_{l}=2(1-\delta)\widetilde{n}/\sqrt{k}$,
  for all $t\geq 0$,
\[
\prob{}{R(\Sigma(K_{\widetilde{n},\widetilde{n}};k)) \leq m_{l}-t}
  \leq 2\,e^{-t^{2}/8m_{u}}.
\]
\end{theorem}

\section{Distribution of symbols in repetition-free LCSs}\label{sec:distrib}
One expects that any symbol is equally likely to show
  up in a repetition-free LCS of two randomly chosen sequences.
Intuitively, this follows from the fact that there is a symmetry between
  symbols.
In fact, one expects something stronger to hold; conditioning on 
  the largest repetition-free LCS being of size~$\ell$, any subset
  of~$\ell$ symbols among the~$k$ symbols of the alphabet
  should be equally likely.
Making this intuition precise is somewhat tricky due to the fact 
  that there might be more than one repetition-free LCS for a given 
  pair of sequences.
The purpose of this section is to formalize the preceding discussion.

First, note that if $G$ is in the support of $\Sigma(K_{n,n};k)$, then each of 
  its connected components is either an isolated node or 
  a complete bipartite graph.
Hence, each connected component
  of~$G$ is in one-to-one correspondence with a symbol from the 
  $k$-ary alphabet.

Now, consider some total ordering, denoted $\preceq$,
  on the noncrossing matchings of $K_{n,n}$. 
For $\ell\in [k]$, let~$\calG_{\ell}$ be the collection of all graphs~$G$ 
  in the support of $\Sigma(K_{n,n};k)$ such that $R(G)=\ell$.
Given $G$ in $\calG_{\ell}$ let 
  $\calC_{\ell}(G)\subseteq [k]$ denote the collection
  of symbols assigned to the nodes of the smallest 
  (with respect to the ordering $\preceq$) 
  noncrossing matching $M$ of $G$ of size $\ell$.
Clearly, the cardinality of $\calC_{\ell}(G)$ is $\ell$.
For $G$ in the support of $\Sigma(K_{n,n};k)$, we say that 
  $M$ is the \emph{canonical} matching of $G$ if $M$ is the 
  smallest, with respect to the ordering $\preceq$, among all largest
  repetition free noncrossing matching of $G$. 
We claim that for $G$ chosen according to 
  $\Sigma(K_{n,n};k)$, conditioned on $R(G)=\ell$, the set of symbols 
  associated to the edges of the canonical matching $M$ of $G$
  is uniformly distributed over all size $\ell$ subsets 
  of $[k]$.
Formally, we establish the following result.
\begin{lemma}\label{lem:overlap}
For all $\ell\in [k]$ and $S\subseteq [k]$ with $|S|=\ell$, 
\[
\prob{}{\calC_{\ell}(G)=S \,\left|\, R(G)=\ell\right.}
  = \frac{1}{\binom{k}{\ell}},
\]
where the probability is taken over the choices of $G$ distributed
  according to $\Sigma(K_{n,n};k)$.
\end{lemma}
\begin{proof}
For a subset $E$ of edges of $K_{n,n}$, define $\calP_{\ell}(E)$
  as the set of elements of $\calG_{\ell}$ whose 
  edge set is exactly~$E$.
Let $\calE_{\ell}$ be the collection of all $E$'s such 
  that $\calP_{\ell}(E)$ is nonempty and
  let $\calP_{\ell}$ be the collection of $\calP_{\ell}(E)$'s
  where~$E$ ranges over subsets of $\calE_{\ell}$.
Observe that $\calP_{\ell}$ is a partition of $\calG_{\ell}$.
Hence, 
\[
\sum_{E\in\calE_{\ell}}\prob{}{E(G)=E \,\left|\, R(G)=\ell\right.}
  = \sum_{E\in\calE_{\ell}}\prob{}{G\in\calP_{\ell}(E) \,\left|\, R(G)=\ell\right.}
  = \prob{}{G\in\calG_{\ell} \,\left|\, R(G)=\ell\right.} = 1.
\]
Moreover,
\[
\prob{}{\calC_{\ell}(G)=S \,\left|\, R(G)=\ell\right.}
  = \sum_{E\in\calE_{\ell}}
    \prob{}{\calC_{\ell}(G)=S \,\left|\, E(G)=E\right.}
    \prob{}{E(G)=E \,\left|\, R(G)=\ell\right.}.
\]
Thus, the desired conclusion will follow immediately once we 
  show that $\prob{}{\calC_{\ell}(G)=S \,\left|\, E(G)=E\right.}
    = 1/\binom{k}{\ell}$ for all $E\in\calE_{\ell}$.
Indeed, let $E\in\calE_{\ell}$ and observe that the condition
  $E(G)=E$ uniquely determines 
  the canonical
  noncrossing matching of $G$ of size $\ell$, say $M=M(G)$.
Moreover, note that any choice of distinct 
  $\ell$ symbols to each of the $\ell$ distinct components of $G$
  to which the edges of $M$ belong is equally likely.
Since there are $\binom{k}{\ell}$ possible choices of $\ell$-symbol
  subsets of $[k]$, the desired conclusion follows.
\end{proof}
The preceding result will be useful in the next section in order to 
  address the following issue.
For $G$ and $G_1,\ldots,G_{\blocks}$ as defined in Section~\ref{sec:lcs},
  suppose that $M_1,\ldots,M_{\blocks}$ are the largest repetition-free
  noncrossing matchings of $G_1,\ldots,G_{\blocks}$, respectively.
As mentioned before
  the union $M$ of the $M_i$'s is a noncrossing matching of $G$, 
  but not necessarily repetition-free.
Obviously, we can remove edges from $M$, keeping one edge for each
  symbol associated to the edges of $M$, and thus obtain a 
  repetition-free noncrossing matching~$M'$ contained in $M$, and thence
  also in $G$.
Clearly, it is of interest to determine the expected number of 
  edges that are removed from $M$ to obtain $M'$, i.e.~$|M\setminus M'|$,
  and in particular whether this number is small.
Lemma~\ref{lem:overlap} is motivated, and will be useful,
  in this context.
The reason being that, conditioning on the size~$s_i$ of the 
  largest repetition-free noncrossing matching in each $G_i$, 
  it specifies the distribution of the set of symbols $\calC_{s_i}(G_i)$
  associated to the edges of the canonical noncrossing matching of $G_i$.
The latter helps in the determination of the sought-after expected value, since
\[
|M\setminus M'| = \sum_{i=1}^{\blocks}\left|\calC_{s_i}(G_i)\right|
  - \left|\cup_{i=1}^{\blocks}\calC_{s_i}(G_i)\right|.
\]

\section{Tail bounds}\label{sec:lower}
In this section we derive bounds on the probability that 
  $R(G)$ is bounded away from its expected value when~$G$ is chosen
  according to $\Sigma(K_{n,n};k)$.
We will ignore the case where $n=O(\sqrt{k})$ 
  due to its limited interest and the impossibility of 
  deriving meaningful asymptotic results. 
Indeed, if $n\leq C\sqrt{k}$ for some positive constant $C$ and 
  sufficiently large $k$, then the expected number of edges 
  of a graph $G$ chosen according to $\Sigma(K_{n,n};k)$
  is $n^{2}/k\leq C^{2}$ (just observe that there 
  are $n^{2}$ potential edges and that each one occurs 
  in $G$ with probability $1/k$).
Since $0\leq R(G)\leq |E(G)|$, when $n=O(\sqrt{k})$, the 
  expected length of a 
  repetition-free LCS will be constant ---
  hence, not well suited for an asymptotic study.
Thus, we henceforth assume that $n=\omega(\sqrt{k})$.
If in addition $n=o(k)$, then Theorem~\ref{theo:small-main}
  already provides the type of tail bounds we 
  are looking for.
Hence, we need only consider the case where $n=\Omega(k)$.
We will show that three different regimes arise.
The first one corresponds to $n=o(k\sqrt{k})$.
For this case we show that the length of a repetition-free LCS 
  is concentrated around its expected value, which in 
  fact is roughly $2n/\sqrt{k}$ (i.e.~the same magnitude as that of
  the length of a standard LCS).
The second one corresponds to $n=\Theta(k\sqrt{k})$.
For this regime we show that the length of a repetition-free 
  LCS cannot be much smaller than a fraction of $k$, 
  and we relate the constant of proportionality with 
  the constant hidden in the asymptotic dependency
  $n=\Theta(k\sqrt{k})$.
The last regime corresponds to $n=(1+\Omega(1))k\sqrt{k}\ln k$.
For this latter case we show that with high probability a
  repetition-free LCS is of size $k$. 

\medskip
Throughout this section, $n$ and $k$ are positive integers, 
  $G$ is a bipartite graph chosen according
  to $\Sigma(K_{n,n};k)$, and $G_1,\ldots,G_{\blocks}$ are as defined
  in~Section~\ref{sec:lcs}, where $b$ is an integer approximately 
  equal to $n/\widetilde{n}$.
Note in particular that $G_i$ is distributed according 
  to $\Sigma(K_{\widetilde{n},\widetilde{n}};k)$. 

\medskip
This section's first formal claim
  is motivated by an obvious fact; if $r=R(G)$ is 
  ``relatively small'', then at least one of the two following situations must
  happen:
\begin{itemize}
\item For some $i\in [\blocks]$, the value of $r_i=R(G_i)$ is 
  ``relatively small''.
\item The sets of symbols, $\calC_{r_i}(G_i)$, associated to the edges of 
  the canonical largest noncrossing matching of $G_i$, for $i\in [\blocks]$, 
  have a ``relatively large'' overlap, more precisely, the cardinality of 
  $\calC_{r}(G)$ is ``relatively small'' compared to 
  the sum, for $i\in [\blocks]$, of the cardinalities of $\calC_{r_i}(G_i)$.
  
\end{itemize}
The next result formalizes the preceding observation.
In particular, it establishes that the probability that $R(G)$ is 
  ``relatively small'' is bounded by the probability that 
  one of the two aforementioned cases occurs (and also gives a 
  precise interpretation to the terms ``relatively large/small'').
\begin{lemma}\label{lem:fact}
Let $b$ be a positive integer.
For $a\geq 0$ and $r\geq t\geq 0$, let 
\begin{align*}
P_1 & = P_1(r,t) \eqdef \!
    \sum_{r_1,\ldots,r_{\blocks}\geq 0\atop r_1+\cdots+r_{\blocks}=\lfloor r-t\rfloor}
    \!\! \prob{}{R(G_{i})\leq r_i, \forall i\in [\blocks]}\,, \tag{Definition of $P_1$}\\
P_2 &  = P_2(a,r,t) \eqdef
    \prob{}{R(G)\leq r-a, \,
    \sum_{i=1}^{\blocks}R(G_i)\geq r-t}. \tag{Definition of $P_2$}
\end{align*}
Then,
\begin{align*}
\prob{}{R(G)\leq r-a} & \leq P_1 + P_2.
\end{align*}
\end{lemma}
\begin{proof}
Just note that
\begin{align*}
\prob{}{R(G)\leq r-a} & = 
    \prob{}{R(G)\leq r-a, 
    \sum_{i=1}^{\blocks}R(G_i)< \lceil r-t\rceil} 
   + 
    \prob{}{R(G)\leq r-a, \,
    \sum_{i=1}^{\blocks}R(G_i)\geq \lceil r-t\rceil} \\
  & \leq 
    \prob{}{\sum_{i=1}^{\blocks}R(G_i)< r-t} +
    \prob{}{R(G)\leq r-a, \,
    \sum_{i=1}^{\blocks}R(G_i)\geq r-t}.
\end{align*}
The desired conclusion follows observing that 
  the two last terms in the preceding displayed expression
  are equal to $P_1$ and $P_2$, respectively.
\end{proof}

The following lemma will be useful in bounding the terms in $P_1$, 
  i.e.~the probability that $R(G_i)$ is ``relatively small'' for some $i$.
Henceforth, for the sake of clarity of exposition, we will ignore
  the issue of integrality of quantities (since we are interested in the 
  case where $n$ is large, ignoring integrality issues should have 
  a negligible and vanishing impact in the following calculations).
\begin{lemma}\label{lem:boundP1}
Let $\delta>0$ and $\widetilde{n}=\widetilde{n}(k)$ be such that it 
  satisfies the hypothesis of Theorem~\ref{theo:small-main}
  and let $m_l=(1-\delta)2\widetilde{n}/\sqrt{k}$.
Let $\blocks=\blocks(k)\eqdef n/\widetilde{n}$.  
Then, 
\[
P_{1} = P_{1}(bm_{l},t) \leq \left(2e(m_{l}+1)\right)^{\blocks}
  \exp\left(-\frac{t^{2}}{16(1+\delta)n/\sqrt{k}}\right).
\]
\end{lemma}
\begin{proof}
Observe that, by independence of the $R(G_i)$'s
  and Theorem~\ref{theo:small-main}, for 
  $m_{u}=(1+\delta)2\widetilde{n}/\sqrt{k}$,
\begin{align*}
& \prob{}{R(G_{i})\leq r_i, \forall i\in [\blocks]} 
   = \prob{}{R(G_{i})\leq m_{l}-(m_{l}-r_i), \forall i\in [\blocks]} \\
& \qquad 
  \leq \prod_{i=1}^{\blocks} \left(2e^{-\max\{0,m_{l}-r_{i}\}^{2}/8m_{u}}\right)
  = 2^{\blocks}e^{-\sum_{i=1}^{\blocks}\max\{0,m_{l}-r_{i}\}^{2}/8m_{u}}.
\end{align*}
By Cauchy-Schwarz, since $\max\{0,x\}+\max\{0,y\}\geq \max\{0,x+y\}$,
  and assuming that $\sum_{i=1}^{\blocks}r_{i}=\lfloor\blocks m_l-t\rfloor$, 
\[
\sum_{i=1}^{\blocks}\max\{0,m_{l}-r_{i}\}^{2}
  \geq \frac{1}{\blocks}\left(\sum_{i=1}^{\blocks}\max\{0,m_{l}-r_{i}\}\right)^{2}
  \geq
 \frac{1}{\blocks}t^{2}.
\]
Recalling that there are $\binom{M+\blocks-1}{\blocks-1}\leq 
  \binom{M+\blocks}{\blocks}$
  ways in which $\blocks$ nonnegative summands can add up to $M\in\NN$, 
  and that $\binom{\mu}{\nu}\leq (e\mu/\nu)^{\nu}$, 
\begin{align*}
P_{1} & \leq  
  \sum_{r_1,\ldots,r_{\blocks}\geq 0\atop r_1+\cdots+r_{\blocks}=\lfloor\blocks m_{l}-t\rfloor}
  \!\! 2^{\blocks}e^{-t^{2}/8\blocks m_{u}} 
  \leq {\lfloor\blocks m_{l}-t\rfloor+\blocks\choose \blocks}2^{\blocks}e^{-t^{2}/8\blocks m_{u}} 
  \leq \left(2e(m_{l}+1)\right)^{\blocks}e^{-t^{2}/8\blocks m_{u}}.
\end{align*}
Since $m_{u}=(1+\delta)2\widetilde{n}/\sqrt{k}$
  and $b\widetilde{n}=n$, the desired conclusion
  follows immediately.
\end{proof}

The next lemma will be useful in bounding $P_{2}$,
  i.e.~the probability 
  that the sets of symbols associated to the edges of 
  the canonical largest noncrossing $G_i$'s matchings have a 
  ``relatively large'' overlap.
The result in fact shows how to translate tail bounds for an 
  urn occupancy model into bounds for $P_2$.
\begin{lemma}\label{lem:boundP2}
If $b$ is a positive integer,
  $a\geq 0$, $r\geq t\geq 0$, and $s=\lceil r-t\rceil$,
  then 
\[
P_{2} = P_{2}(a,r,t)  \leq \prob{}{k-Y^{(k,s)} \leq r-a}.
\]
\end{lemma}
\begin{proof}
Clearly,
\begin{align*}
P_{2} & = 
    \sum_{s_1,\ldots,s_{\blocks}\geq 0\atop s_1+\cdots+s_{\blocks}\geq r-t}
    \prob{}{R(G)\leq r-a 
      \,\left|\, R(G_i)= s_i, \forall i\in [\blocks]\right.}
      \prob{}{R(G_i)= s_i, \forall i\in [\blocks]}.
\end{align*}
Let $\calC_{\ell}(\cdot)$ be as defined in Section~\ref{sec:distrib}.
Note that if we take the union of noncrossing matchings, one $M_i$
  for each $G_i$, we get a noncrossing matching $M=\cup_i M_i$ of $G$. 
However, the edges of $M$
do not necessarily have distinct associated symbols. 
By throwing away all but one of the edges of $M$ 
  associated to a given symbol, one obtains a repetition-free
  noncrossing matching of $G$.
It follows that, conditioning on $R(G_{i})=s_{i}$ for all $i\in [\blocks]$,
\[
R(G) \geq \left|\bigcup_{i=1}^{\blocks}\calC_{s_i}(G_i)\right|.
\]
Thus,
\begin{align*}
\prob{}{R(G)\leq r-a 
      \,\left|\, R(G_i)= s_i, \forall i\in [\blocks]\right.}
  & \leq 
\prob{}{\left|\bigcup_{i=1}^{\blocks}\calC_{s_i}(G_i)\right|\leq r-a 
      \,\left|\, R(G_i)= s_i, \forall i\in [\blocks]\right.} \\
  & = 
\prob{}{\left|\bigcup_{i=1}^{\blocks}\calC_{s_i}(G_i)\right|\leq r-a 
      \,\left|\, \left|\calC_{s_{i}}(G_i)\right|=s_i, \forall i\in [\blocks]\right.}.
\end{align*}
Let $\vec{s}=(s_1,\ldots,s_b)$.
We claim that $\left|\bigcup_{i=1}^{\blocks}\calC_{s_i}(G_i)\right|$
  conditioned on $\left|\calC_{s_{i}}(G_i)\right|=s_i$, for all $i\in [\blocks]$,
  is distributed exactly as the number of nonempty urns left
  when the Grouped Urn $(k,\vec{s})$-model 
  (as defined in Section~\ref{sec:distrib}) ends,
  i.e.~is distributed as the random variable $k-X^{(k,\vec{s})}$ 
  (where $X^{(k,\vec{s})}$ is as defined in Section~\ref{sec:distrib}).
Indeed, it suffices to note that by Proposition~\ref{prop:regimes}, 
  conditioned on $\left|\calC_{s_{i}}(G_i)\right|=s_i$, the 
  set $S_i=\calC_{s_i}(G_i)$ is a randomly and uniformly chosen
  subset of $[k]$ of size $s_i$, and that $k-X^{(k,\vec{s})}$
  is distributed exactly as 
  $\left|\bigcup_{i=1}^{\blocks}\calC_{s_i}(G_i)\right|$. 
It follows, from the forgoing discussion and Lemma~\ref{lem:dominance}, that
\[
\prob{}{R(G)\leq r-a 
      \,\left|\, R(G_i)= s_i, i\in [\blocks]\right.}
  = \prob{}{k-X^{(k,\vec{s})}\leq r-a}
  \leq \prob{}{k-Y^{(k,s)}\leq r-a}. \qedhere
\]
\end{proof}

The next result establishes the first of the announced tail bounds, for
  the first of the three regimes indicated at the start of this section.
An interesting aspect, that is not evident from the theorem's statement, 
  is the following fact that is implicit in its proof; if the speed of 
  growth of $n$ as a function of $k$ is not too fast, then we may choose
  $\blocks$ as a function of $k$ so that $\sum_{i=1}^{\blocks}R(G_i)$ is 
  roughly (with high probability) equal to~$R(G)$.
In particular, the proof argument rests on the fact that, for an 
  appropriate choice of parameters, the 
  canonical largest noncrossing matching of $G_i$ is of size approximately
  $2(n/\blocks)/\sqrt{k}$, and with high probability 
  there is very little overlap between the symbols associated to
  the edges of the canonical largest noncrossing 
  matchings of distinct $G_i$'s.
\begin{theorem}\label{theo:small-regime}
If $n=o(k\sqrt{k})$, then for every $0<\xi\leq 1$
  there is a sufficiently large constant $k_{0}=k_{0}(\xi)$ such that,
  for all $k>k_{0}$,
\[
\prob{}{R(G) \leq (1-\xi)2n/\sqrt{k}} \leq 2e^{-\frac{1}{10}\xi^22n/\sqrt{k}}.
\]
\end{theorem}
\begin{proof}
Let $c>1$ be large enough so $(1-1/c)^2\geq (9/10)(1+\xi/c)$.
Let $\delta=\xi/c$ and $t=(1-1/c)\xi 2n/\sqrt{k}$.
Now, choose $\widetilde{n}=\widetilde{n}(k)=k^{3/4}$ (instead of $3/4$,
  any exponent strictly between $1/2$ and $1$ suffices).
Note that one can choose $\widetilde{k}_0$
  (depending on $\xi$ through $\delta$) 
  so that for all $k\geq \widetilde{k}_0$ the conditions on
  $\widetilde{n}$ of Theorem~\ref{theo:small-main}
  are satisfied.
Let $m_l$ and $m_u$ be as in Theorem~\ref{theo:small-main}.
Note that $m_{l}=(1-\xi/c)2\tilde{n}/\sqrt{k}=\Theta(k^{1/4})$,
  $b=n/\tilde{n}=n/k^{3/4}$, and
  $bm_{u}=(1+\xi/c)2n/\sqrt{k}$.
Hence, by Lemma~\ref{lem:boundP1},
\[
P_{1} \leq \exp\left(b\ln(2e(m_{l}+1)))-\frac{t^2}{8bm_{u}}\right)
  =
  \exp\left(\frac{n}{\sqrt{k}}\Theta\left(k^{-1/4}\ln k\right)
     -\frac{(1-1/c)^{2}\xi^2 2n/\sqrt{k}}{8(1+\xi/c)}\right).
\]
Since $k^{-1/4}\ln k=o(1)$ 
  and $(1-1/c)^2\geq (9/10)(1+\xi/c)$,
  it follows that for a sufficiently large $k'_0\geq\widetilde{k}_0$ 
  it holds that for all $k\geq k'_{0}$, 
\[
P_{1} 
  \leq \exp\left(-\frac{(1-1/c)^{2}\xi^22n/\sqrt{k}}{9(1+\xi/c)}\right)
  \leq \exp\left(-\frac{1}{10}\xi^2 2n/\sqrt{k}\right).
\]
On the other hand, since 
  $t=(1-1/c)\xi 2n/\sqrt{k}$, if we fix $a=\xi 2n/\sqrt{k}$, then
  we have that $t-a\leq -(\xi/c)2n/\sqrt{k}$.
Taking 
  $s = bm_{l}-t = (1-\xi)2n/\sqrt{k}\leq 2n/\sqrt{k}$, 
  as $\xi\leq 1$, by Lemma~\ref{lem:boundP2} and 
  Proposition~\ref{prop:regimes}, Part~\ref{it:small-regime}, 
\begin{align*}
P_2 & \leq \prob{}{k-Y^{(k,s)} \leq bm_{l}-a} 
  = \ \prob{}{k-Y^{(k,s)} \leq s+t-a} \\
  & \leq \prob{}{k-Y^{(k,s)} \leq s-\frac{\xi 2n}{c\sqrt{k}}} 
  \leq \left(\frac{2cen}{\xi k\sqrt{k}}\right)^{(\xi/c)2n/\sqrt{k}}.
\end{align*}
Let $k''_0$ be sufficiently large (depending on $\xi$) so that
  $2cen/(\xi k\sqrt{k})\leq e^{-c\xi/10}$ for all $k \geq k''_0$ 
  (such a $k''_0$ exists because $n=o(k\sqrt{k})$).
It follows that for $k\geq k''_0$ we can upper bound 
  $P_2$ by $\exp\left(-\frac{1}{10}\xi^2 2n/\sqrt{k}\right)$.

Since by Lemma~\ref{lem:fact} we know that
  $\prob{}{R(G)\leq (1-\xi)2n/\sqrt{k}}\leq P_1+P_2$, it follows that for 
  $k\geq k_0=k_{0}(\xi)\eqdef\max\{k'_0,k''_0\}$ we get the claimed bound.
\end{proof}

Next, we consider a second regime, but first we establish an
  inequality that we will soon apply.
\begin{claim}\label{claim}
For every $0\leq x\leq 1$ and $\rho\geq 0$,
\[
e^{-\rho(1-x)}-e^{-\rho}-x(1-e^{-\rho}) \leq 0.
\]
\end{claim}
\begin{proof}
Since $0\leq x\leq 1$, it holds that 
  $0\leq x^n\leq x$ for all $n\in\NN\setminus\setof{0}$.
Then, since $e^{y}=\sum_{n\in\NN} y^n/n!$ and performing some basic 
  arithmetic,
\[
e^{\rho x}-1 - x(e^{\rho}-1) = \sum_{n\geq 1} 
  \frac{\rho^n}{n!}\left(x^n-x\right)\leq 0.
\]
Multiplying by $e^{-\rho}$, the claimed result immediately follows.
\end{proof}

\begin{theorem}\label{theo:middle-regime}
Let $\rho > 0$ and $0<\xi<1$.
If $n=\frac{1}{2}\rho k\sqrt{k}$, then there 
  is a sufficiently large constant $k_{0}=k_{0}(\rho,\xi)$ such that 
  for all $k>k_{0}$,
\[
\prob{}{R(G) \leq (1-\xi)k(1-e^{-\rho})} 
    \leq 2e^{-\frac{\xi^{2}}{32(1+\xi/12)}k(1-e^{-\rho})}
    \leq 2e^{-\frac{1}{35}\xi^{2}k(1-e^{-\rho})}.
\]
\end{theorem}
\begin{proof}
Since $\xi<1$,
  the second stated inequality follows immediately from the first one.
We thus focus on establishing the first stated inequality. 
 
Let $\delta=\xi/12$. 
Now, choose $\tilde{n}=k^{3/4}$ (instead of $3/4$,
  any exponent strictly between $1/2$ and $1$ suffices)
  and set $b=n/\widetilde{n}$.
Note that one can choose $k'_0$
  (depending on $\xi$ through $\delta$) 
  so that for all $k>k'_0$ the conditions on
  $\tilde{n}=\tilde{n}(k)$ of Theorem~\ref{theo:small-main}
  are satisfied.
Let $m_{l}= (1-\delta)2\widetilde{n}/\sqrt{k}$ and observe that 
  $bm_{l}=(1-\delta)2n/\sqrt{k}=(1-\xi/12)\rho k$.
Choose $t=(2\xi/3)\rho k$ and note that 
  $s= bm_{l}-t = (1-3\xi/4)\rho k$.
Let $\lambda = \expec{}{Y^{(k,s)}}=k(1-1/k)^{s}$ be 
  as in Proposition~\ref{prop:regimes}.
We claim that for $0< \xi < 1$,
\begin{align}\label{eqn:lambda-bound}
\lambda & \leq ke^{-\rho} + (3\xi/4) k(1-e^{-\rho}).
\end{align}
Indeed, since $1+x\leq e^{x}$ we have that 
  $\lambda=k(1-1/k)^{s}\leq ke^{-\rho(1-3\xi/4)}$, 
  so to prove~\eqref{eqn:lambda-bound}
  it suffices to recall that by Claim~\ref{claim} we have that
  $e^{-\rho(1-3\xi/4)} \leq e^{-\rho}+(3\xi/4)(1-e^{-\rho})$.

Now, fix $\tilde{a}$ and $a$ so $\tilde{a}= \xi k(1-e^{-\rho})$ and 
  $a=bm_{l}-k(1-e^{-\rho})+\tilde{a}$.
By Lemma~\ref{lem:boundP2} and~\eqref{eqn:lambda-bound},
\begin{align*}
P_{2} &
  \leq \prob{}{k-Y^{(k,s)} \leq bm_{l}-a}
  = \prob{}{Y^{(k,s)} \geq ke^{-\rho}+\tilde{a}}
  \leq \prob{}{Y^{(k,s)} \geq \lambda +(\xi/4)k(1-e^{-\rho})}.
\end{align*}
Hence, taking $p=\lambda/k\leq 1$, $q=1-p$,
  and applying Proposition~\ref{prop:regimes},
  Part~\ref{it:middle-regime}, 
\begin{align*}
P_{2}  & 
  \leq \exp\left(-\frac{\xi^{2}(1-e^{-\rho})^{2}k}{32(pq+(\xi/12)(1-e^{-\rho}))}\right)
  \leq \exp\left(-\frac{\xi^{2}(1-e^{-\rho})^{2}k}{32(q+(\xi/12)(1-e^{-\rho}))}\right).
\end{align*}
Again by Proposition~\ref{prop:regimes}, we know that $\lambda=k(1-1/k)^{s}$.
Thus, recalling that by our choice of parameters $s=(1-3\xi/4)\rho k$
  and since $(1-1/k)^{s}=(1-1/k)^{\rho(1-3\xi/4)k}$ 
  converges to $e^{-\rho(1-3\xi/4)}>e^{-\rho}$ when 
  $k$ goes to $\infty$, it follows 
  that $q=1-p=1-(1-1/k)^{s}$ can be 
  upper bounded by $1-e^{-\rho}$ for all $k>k''_{0}$
  and some sufficiently large $k''_0>k'_0$ (depending on $\xi$).
Hence, for $k>k''_0$ it holds that 
  $q+(\xi/12)(1-e^{-\rho})\leq (1+\xi/12)(1-e^{-\rho})$,
  and
\begin{align*}
P_{2}  & 
  \leq \exp\left(-\frac{\xi^{2}}{32(1+\xi/12)}k(1-e^{-\rho})\right).
\end{align*}
We will now upper bound $P_1$.
Note that, by our choice for $\widetilde{n}$ and 
  the hypothesis on $n$, we have 
  $m_{l}=(1-\xi/12)2\tilde{n}/\sqrt{k}=(1-\xi/12) k^{1/4} \leq k$,
  $b=n/\tilde{n}\leq\rho k^{3/4}$, and
  $bm_{u}=(1+\xi/12)2n/\sqrt{k}=(1+\xi/12)\rho k$.
Recalling that we fixed $t=(2\xi/3)\rho k$,
  by Lemma~\ref{lem:boundP1},
\[
P_{1} \leq \exp\left(b\ln(2e(m_{l}+1)))-\frac{t^2}{8bm_{u}}\right)
  \leq 
  \exp\left(\rho k^{3/4}\ln(2e(k+1)))-\frac{\xi^2 \rho k}{18(1+\xi/12)}\right).
\]
Since $k^{3/4}\ln(2e(k+1))=o(k)$ and because
  $1-e^{-\rho}\leq \rho$, 
  it follows that for a sufficiently large $k'''_0$ it holds that 
  for all $k> k'''_{0}$, 
\[
P_{1} 
  \leq \exp\left(-\frac{\xi^2}{19(1+\xi/12)}\rho k\right)
  \leq \exp\left(-\frac{\xi^2}{19(1+\xi/12)}k(1-e^{-\rho})\right).
\]
Since 
  $\prob{}{R(G)\leq (1-\xi)k(1-e^{-\rho})}\leq P_1+P_2$ for 
  $k>k_0=k_{0}(\rho,\xi)\eqdef\max\{k''_0,k'''_0\}$, 
  we get the claimed bound.
\end{proof}

Our next result establishes that if $n$ is sufficiently large with respect
  to $k$, then with high probability 
  the repetition-free LCS is of size $k$, 
  i.e.~it is a permutation of the underlying alphabet.
Moreover, the theorem's proof implicitly shows something stronger;
  if the speed of 
  growth of $n$ as a function of $k$ is fast enough, then we may choose
  $\blocks$ as a function of $k$ so that with 
  high probability every symbol of the 
  $k$-ary alphabet show up in association to an edge of
  a canonical maximum size matching of some $G_i$ --- chosen such 
  edges one obtains a noncrossing repetition free matching of $G$ of 
  the maximum possible size $k$.
\begin{theorem}\label{theo:large-regime}
If $n=(\frac{1}{2}+\xi)k\sqrt{k}\ln k$ for some $\xi>0$, then
  there is a sufficiently large constant $k_{0}=k_{0}(\xi)$ such that 
  for all $k>k_{0}$,
\[
\prob{}{R(G)\neq k} \leq \frac{2}{k^{\xi}}.
\]
\end{theorem}
\begin{proof}
Let $\delta=\delta(\xi)>0$ be such that
  $(1-\delta)(1+2\xi)= 1+3\xi/2$.
Now, let $\widetilde{n}=k^{3/4}$ (instead of 
  $3/4$, any exponent strictly between $1/2$ and $1$ suffices)
  and set $b=n/\widetilde{n}$.
Note that one can choose $k'_{0}$ (depending 
  on $\xi$ through $\delta$) so that for all $k>k'_{0}$ the conditions 
  on $\widetilde{n}=\widetilde{n}(k)$  
  of Theorem~\ref{theo:small-main} are satisfied.
Let $m_{l}= (1-\delta)2\widetilde{n}/\sqrt{k}$ be as in 
  Theorem~\ref{theo:small-main}.
Observe that $\blocks m_{l}= (1-\delta)2n/\sqrt{k} = (1+3\xi/2)k\ln k$.
Choose $t=(\xi/2)k\ln k$ so that 
  $s= \blocks m_{l}-t=(1+\xi)k\ln k$.
Fix $a$ so $k-\blocks m_{l}+a=1$.
By Lemma~\ref{lem:boundP2} and Proposition~\ref{prop:regimes}, Part~\ref{it:large-regime},
\begin{align*}
P_{2} &   \leq \prob{}{Y^{(k,s)} \geq k-bm_{l}+a} 
  = \prob{}{Y^{(k,s)} \neq 0} 
  \leq \frac{1}{k^{\xi}}.
\end{align*}
By the hypothesis on $n$ and the choice of $\widetilde{n}$, we have that 
  $b=n/\tilde{n}=(\frac{1}{2}+\xi)k^{3/4}\ln k$,
  so recalling that $m_{l}=(1-\delta)2\widetilde{n}/\sqrt{k}=(1-\delta)2k^{1/4}$,
\begin{align*}
b\ln(2e(m_{l}+1)) & = O(k^{3/4}\ln^{2} k) = o(k\ln k).
\end{align*}
Furthermore, let $m_{u}= (1+\delta)2\widetilde{n}/\sqrt{k}$ be as in 
  Theorem~\ref{theo:small-main}.
Thus,
\begin{align*}
\frac{t^2}{16(1+\delta)n/\sqrt{k}} 
  & = \frac{t^{2}}{8bm_{u}} = \frac{\xi^{2}k\ln k}{32(1+\delta)(1+2\xi)}.
\end{align*}
Hence, by Lemma~\ref{lem:boundP1}, for a sufficiently large constant $k''_{0}$
  (again depending on $\xi$ through $\delta$),
  we can guarantee that, for all $k>k''_{0}$,
\[
P_{1}\leq (2e(m_{l}+1))^{b}e^{-t^{2}/8bm_{u}}
  = \exp\left(b\ln(2e(m_{l}+1))-\frac{t^{2}}{8bm_{u}}\right) 
  \leq \frac{1}{k^{\xi}}.
\] 
Summarizing, for $k>k_0=k_{0}(\xi)\eqdef\max\{k'_0,k''_0\}$, we get that 
  $\prob{}{R(G)\neq k}\leq P_1+P_2\leq 2/k^{\xi}$.
\end{proof}

{From} the lower tail bounds for $R(G)$ obtained above, one can easily
  derive lower bounds on the expected value of $R(G)$ via the 
  following well-known trick.
\begin{lemma}
If $X$ is a nonnegative random variable and $x>0$, then
  $\displaystyle\expec{}{X} \geq  x\left(1-\prob{}{X\leq x}\right)$.
\end{lemma}
\begin{proof}
Let $\indic_{A}$ denote the indicator of the event $A$ occurring. 
Just observe that
\begin{align*}
\expec{}{X} & = 
  \expec{}{X\indic_{\setof{X\leq x}}} + \expec{}{X\indic_{\setof{X> x}}} 
  \geq 
  x\,\expec{}{\indic_{\setof{X> x}}} 
  = 
  x\left(1-\prob{}{X\leq x}\right). \qedhere
\end{align*}
\end{proof}

Theorem~\ref{th:main} now follows 
  as a direct consequence of the preceding lemma, 
  Theorems~\ref{theo:small-regime}, \ref{theo:middle-regime},
  and~\ref{theo:large-regime}, 
  and the fact that $R(G)\leq k$.

\vspace{5mm}

\subsection*{Acknowledgements}

The authors would like to thank Carlos E.~Ferreira, Yoshiharu Kohayakawa, 
and Christian~Tjandraatmadja for some discussions in the preliminary 
stages of this work. 

\bibliographystyle{plain}
\bibliography{biblio}

\end{document}